\documentclass[reqno,b5paper]{amsart}
\usepackage{amsmath}
\usepackage{amssymb}
\usepackage{amsthm}
\usepackage{enumerate}
\usepackage[mathscr]{eucal}
\theoremstyle{plain}
\newtheorem{thm}{Theorem}[section]
\newtheorem{cor}[thm]{Corollary}

\newtheorem{Example}{Example}[section]
\newtheorem{note}{Note}[section]

\theoremstyle{definition}
\newtheorem{defn}{Definition}[section]
\newtheorem{rem}{Remark}[section]

\begin{document}

\setcounter {page}{1}
%---------------Title,Author,Abstract-----------------------------------------------
\title{ A Generalized notion of convergence of sequences of subspaces in an
inner product space via ideals }

\author[ P. Malik, S. Das ]{Prasanta Malik* and Saikat Das*\ }
\newcommand{\acr}{\newline\indent}
\maketitle
\address{{*\,} Department of Mathematics, The University of Burdwan, Golapbag, Burdwan-713104,
West Bengal, India.
                Email: pmjupm@yahoo.co.in, dassaikatsayhi@gmail.com \acr
           }

\maketitle

\begin{abstract}
In this paper we introduce the notion of $\mathcal{I}$-convergence of sequences of $k$-dimensional
subspaces of an inner product space, where $\mathcal{I}$ is an ideal of subsets of $\mathbb{N}$,
the set of all natural numbers and $k\in\mathbb N$. We also study some basic properties of this notion.
\end{abstract}

\author{}
\maketitle { Key words and phrases : Ideal, filter, $\mathcal{I}$-convergence, inner product space, $n$-norm.} \\


%-------------------------Section 1- Background and introduction-----------------------

\section{\textbf{Introduction and background:}}
The notion of convergence of sequences of points was extended to the notion of convergence of sequences of
sets by many authors \cite{ So, Wj1, Wj2}. Manuharawati et al. \cite{Ma1}, \cite{Ma2} introduced the
concepts of convergences of sequences of 1-dimensional and 2-dimensional
subspaces of a normed linear space and in \cite{Ma3} they have introduced the convergence notion of
sequences of $k$-dimensional ($k\in\mathbb{N}$) subspaces of an inner product space $\mathcal X$ of
dimension $k$ or higher ( may be infinite).

On the other hand the notion of convergence of sequences of real numbers was extended to the
notion of statistical convergence by Fast \cite{Fa} (and also independently by Schoenberg \cite{Sc} ) with the
concept of natural density. A lot of work have been done in this direction, after the work of
Salat \cite{Sl} and Fridy \cite{Fr1}. For more primary works in this field one can see  \cite{ Fr2, Ko1} etc.

The notion of statistical convergence of sequences of real numbers further extended to the notion
of $\mathcal{I}$-convergence by Kostyrko et al. \cite{Ko2}, with the notion of an ideal $\mathcal{I}$ of subsets
of the set of all natural numbers $\mathbb{N}$. For more works in this direction one can see \cite{ Dm, La}.

Recently in 2021, using the notion of natural density, F. Nuray \cite{Nu} has introduced and studied the notion
of statistical convergence of a sequence of $k$-dimensional ($k\in\mathbb{N}$) subspaces of an inner product
space. It seems therefore reasonable to introduce and study the notion of $\mathcal{I}$-convergence of a sequence
of $k$-dimensional subspaces of an inner product space. In this paper we do the same and investigate some basic
properties of this notion. Our results extend the results of Nuray \cite{Nu} and Manuharawati et al. \cite{Ma3}.

%---------------------------Section 2- Definitions and notations-------------------------------------

\section{\textbf{ Basic Definitions and Notations:}}

We first recall some basic definitions and notation related to inner product space from the literature. Throughout the
paper $\mathcal X$ stands for an inner product space $(\mathcal{X},<. , .>)$ of
dimension $k$ ($k\in\mathbb{N}$) or higher (may be infinite) over the field $\mathbb F~~(= \mathbb R~~ or ~~\mathbb C)$,
$U_{n}~~(n \in\mathbb{N})$, $V$ and $W$ denote $k$-dimensional subspaces of $\mathcal X$. Also throughout the paper, it
is supposed that $U_{n} = span \{u^{(n)}_{1}, u^{(n)}_{2},..., u^{(n)}_{k}\}~~~ (n \in\mathbb{N})$, $V = span \{v_{1}, v_{2},..., v_{k}\}$
and $W= span \{w_{1}, w_{2},..., w_{k}\}$, where $ \{ u^{(n)}_{1}, u^{(n)}_{2}, ..., u^{(n)}_{k} \} $ is an orthonormal basis of $U_{n}~~(n \in\mathbb{N})$ and
$\{v_{1}, v_{2},..., v_{k}\}$, $\{w_{1}, w_{2},..., w_{k}\}$ are respective orthonormal bases of $V$ and $W$. Also for
any $x\in \mathcal{X}, \left\|x\right\|= \sqrt{\left\langle x,x\right\rangle}$.

The orthogonal projection of a vector $u \in\mathcal{X}$ onto the subspace $V$ is denoted by $P_{V}(u)$ and is defined by
\begin{align*}
P_{V}(u) = \sum^{k}_{j=1}<u, v_{j}> v_{j}.
\end{align*}

The distance between a vector $ u \in \mathcal{X} $ and a subspace $V$ of $\mathcal{X}$ is denoted by $\widetilde{d} (u,V)$
and is defined by
\begin{align*}
  \widetilde{d} (u,V) =& \inf\{\left\|u-v\right\| : v\in V \} \\
      =& \left\|u-P_{V}(u)\right\|
\end{align*}
where $P_{V}(u)$ is the projection of the vector $u$ upon $V$.

Distance or gap between two subspaces $U$ and $V$ of $\mathcal{X}$ (see \cite{ Ka, Ma3}) is denoted by $\widetilde{d}(U,V)$ and is
defined by
\begin{align*}
\widetilde{d}(U,V)=&\sup \{~~ \inf \{\left\|u-v\right\|: v\in V\} : u\in U, \left\|u\right\|=1\}
\\=& \displaystyle{ \sup_{u\in U, \left\|u\right\|=1}}\left\|u-P_{V}(u)\right\| .
\end{align*}

Using the notion of distance between two subspaces, in (\cite{Ma3}) Manuharawati et al. introduced the notion of convergence
of a sequence of $k$-dimensional subspaces of an inner product space $(\mathcal{X},<. , .>)$.

\begin{defn}(\cite{Ma3}):
Let $\{U_{n}\}_{n\in \mathbb N}$ and $V$ be $k$-dimensional subspaces of $\mathcal X$. Then the sequence $\{U_{n}\}_{n\in\mathbb{N}}$ is
said to converge to the subspaces $V$, written as
$\displaystyle{ \lim_{n\rightarrow \infty} } U_{n}= V $, if
\begin{align*}
&\displaystyle{\lim_{n\rightarrow \infty} }\widetilde{d} (U_{n},V)= 0
\\ i.e.~~~ &\displaystyle{\lim_{n\rightarrow \infty}\sup_{u\in U_{n}, \left\|u\right\|=1}}\left\|u - P_{V}(u)\right\|= 0.
\end{align*}
\end{defn}

On the other hand the usual notion of convergence of real sequences was extended to the notion of statistical
convergence \cite{ Fa, Sc} using the concept of natural density.

\begin{defn}(\cite{Fr1}):
Let $\mathcal P$ be a subset of $\mathbb{N}$. The quotient
$d_{j}(\mathcal P)= \frac{\left|\mathcal P\cap
\{1,2,...,j\}\right|}{j} $ is called the $j^{th}$ partial density
of $\mathcal P$, for all $j\in \mathbb{N}$. Now the limit, $
d(\mathcal P)= \displaystyle{\lim_{j\rightarrow
\infty}}d_{j}(\mathcal P)$ (if it exists) is called the natural
density or simply density of $\mathcal P$.
\end{defn}

\begin{defn}(\cite{Fr1}):
A sequence $ \{x_{n}\}^{\infty}_{n=1} $ of real numbers is said to be statistically convergent to $x (\in\mathbb{R})$,
if for any $\epsilon > 0 $, $d(\mathcal A(\epsilon)) = 0$, where $\mathcal A(\epsilon) = \{m\in\mathbb{N}: \left|x_{m}-x\right|\geq \epsilon\}$.
\end{defn}

Further the notion of statistical convergence was extended to the notion of $\mathcal I$-convergence by
Kostyrko et al.(\cite{Ko2}) using the concept of an ideal $\mathcal I$ of subsets of $\mathbb N$.

\begin{defn}(\cite{Ko2}):
Let $\mathcal D$ be a non-empty set. A non-empty class $\mathcal I$ of subsets of $\mathcal D$ is said to be an
ideal in $\mathcal D$, provided $\mathcal I$ satisfies the conditions
  $(i)~~ \phi \in\mathcal I$, $(ii)$ if $\mathcal{A}, \mathcal{B} \in\mathcal I$ then $\mathcal{A}\cup \mathcal{B} \in\mathcal I$ and $(iii)$
    if $\mathcal{A}\in \mathcal I$ and $\mathcal{B}\subset \mathcal{A}$ then $\mathcal{B}\in \mathcal I$.
\end{defn}

An ideal $\mathcal I$ in $\mathcal{D}~~ (\neq \phi)$ is called non-trivial if $\mathcal {D} \notin \mathcal I$ and $\mathcal I \neq \{\phi\}$.

An ideal $\mathcal I$ in $\mathcal{D}~~ (\neq \phi)$ is called an admissible ideal if $\{z\} \in{\mathcal I},~~~\forall z \in\mathcal D$.

Throughout the paper we take $\mathcal I$ as a non-trivial admissible ideal in $\mathbb{N}$, unless otherwise mentioned.

\begin{defn}(\cite{Ko2}):
A non-empty class $\mathcal F$ of subsets of $\mathcal D~~(\neq \phi)$ is said to be a filter
in $\mathcal D$, provided $(i)~~  \phi \notin \mathcal F$, $(ii)$ if $\mathcal A, \mathcal B \in\mathcal F$
then $\mathcal A\cap \mathcal B \in\mathcal F$
and $(iii)$ if $\mathcal A\in\mathcal F$ and $\mathcal B$ is a subset of $\mathcal D$ such that $\mathcal B \supset \mathcal A$
then $\mathcal B \in\mathcal F$.
\end{defn}

Let $\mathcal I$ be a non-trivial ideal in $\mathcal D$. Then $\mathcal{F(I)}= \{ \mathcal{D-A} : \mathcal A \in\mathcal I \}$ forms
a filter on $\mathcal D$, called the filter associated with the ideal $\mathcal I$.

\begin{defn}(\cite{Ko2}):
A sequence $\{x_{n}\}^{\infty}_{n=1}$ of real numbers is said to be $\mathcal I$-convergent to $x \in \mathbb{R}$ if
$\forall~~~ \epsilon > 0$, the set $\mathcal A(\epsilon)= \{n \in\mathbb{N}: \left| x_{n}- x \right|\geq \epsilon\} \in\mathcal I$
or in other words for each $\epsilon > 0$, $\exists ~~~\mathcal B(\epsilon)\in \mathcal F(I)$ such that
$ \left|x_{n}- x\right| < \epsilon, \forall~~~ n \in\mathcal B(\epsilon)$.\\
In this case we write $\mathcal I-\displaystyle\lim_{n\rightarrow \infty}x_{n}= x$.
\end{defn}

In (\cite{Nu}) the concept of statistical convergence of sequences of subspaces of an inner product space was introduced by F. Nuray as follows:

\begin{defn}(\cite{Nu}):
A sequence $\{U_{n}\}_{n\in \mathbb N}$ of $k$-dimensional subspaces of $\mathcal X$, is said to converge
statistically to a $k$-dimensional subspace $V$ of $\mathcal X$ if for every $\epsilon > 0,~~ d(\mathcal A(\epsilon))= 0 $, where
\begin{align*}
\mathcal A(\epsilon) =\{n\in \mathbb N: \widetilde{d}(U_{n}, V)\geq \epsilon \}
=\{n\in \mathbb N:\displaystyle{\sup_{u\in U_{n}, \left\|u\right\|=1}}\left\|u- P_{V}(u)\right\|\geq \epsilon\}.
\end{align*}
\end{defn}
In this case we write $st-\displaystyle\lim_{n\rightarrow \infty} U_{n}= V$.

In this paper we extend the above concept of statistical convergence of sequences of subspaces
of an inner product space to $\mathcal I$-convergence and study some fundamental properties of
this notion. We also establish some equivalent condition of this convergence notion.

%------------------------Section 3- Main Results------------------------------------
\section{\textbf{ Main Results:}}

\begin{defn}:
A sequence $\{U_{n}\}_{n\in \mathbb N}$
of $k$-dimensional subspaces of $\mathcal X$ is said to be $\mathcal I$-convergent to a $k$-dimensional subspace $V$ if
$\forall~~ \epsilon > 0,~~ \mathcal A(\epsilon)= \{n\in \mathbb N: \displaystyle{\sup_{u\in U_{n}, \left\|u\right\|=1}}
\left\|u- P_{V}(u)\right\|\geq \epsilon\} \in \mathcal I$.
\\In other words $\{U_{n}\}_{n\in \mathbb N}$ is said to $\mathcal I$-converges to $V$ if $\forall~~ \epsilon > 0, \exists~~
\mathcal B(\epsilon)\in \mathcal F(\mathcal I)$ such that
\begin{align*}
\widetilde{d}(U_{n}, V)< \epsilon, \forall~~ n\in \mathcal B(\epsilon)
\\ i.e. \displaystyle{\sup_{u\in U_{n}, \left\|u\right\|=1}} \left\|u- P_{V}(u)\right\|< \epsilon, \forall~~ n\in \mathcal B(\epsilon).
\end{align*}
\end{defn}
In this case we write $\mathcal I-\displaystyle{\lim_{n\rightarrow \infty}} U_{n}= V$ and $V$ is called
an $\mathcal I$-limit of the sequence $\{U_{n}\}_{n\in \mathbb N}$.

\begin{thm}
$\mathcal{I}-\displaystyle\lim_{n\rightarrow \infty}U_{n}= V$ if and only if $\mathcal{I}-\displaystyle\lim
_{n\rightarrow \infty}\left\|u^{(n)}_{i}- P_{V}(u^{(n)}_{i})\right\|= 0, \forall~~~ i= 1, 2,...,k$.
\end{thm}

\begin{proof} First, let $\mathcal{I}-\displaystyle\lim_{n\rightarrow \infty}U_{n}= V$  and let $\epsilon > 0$ be given.
Then $\exists~~ \mathcal A(\epsilon)\in \mathcal{F(I)}$ such that
\begin{align*}
d(U_{n},V)=\displaystyle\sup_{u\in U_{n}, \left\|u\right\|=1}\left\|u- P_{V}(u)\right\|< \epsilon ,\forall~~ n\in \mathcal A(\epsilon) ~~~~~~~~~~~~~~~~~\ldots
\ldots (1)
\end{align*}
\\Since for each $n\in \mathbb{N}$, $\{u^{(n)}_{1}, u^{(n)}_{2},...,u^{(n)}_{k}\}$ is an orthonormal
basis for $U_{n}$, so $\left\|u^{(n)}_{i}\right\|= 1$ for each $i=1,2,...,k$ and hence from (1) we have $\forall~~ n\in A(\epsilon)$,
\begin{align*}
&\left\|u^{(n)}_{i}- P_{V}(u^{(n)}_{i})\right\|\leq\displaystyle\sup_{u\in U_{n}, \left\|u\right\|=1}
\left\|u- P_{V}(u)\right\|< \epsilon, \forall~~~ i=1,2,...,k
\\ \Rightarrow &    \left\|u^{(n)}_{i}- P_{V}(u^{(n)}_{i})\right\|< \epsilon,   \forall~~~ i=1,2,...,k.
\end{align*}
Therefore, $\forall~~~ i=1,2,...,k$,
\begin{align*}
 & \mathcal A(\epsilon)\subset \{n\in \mathbb{N}: \left\|u^{(n)}_{i}- P_{V}(u^{(n)}_{i})\right\|< \epsilon \}
\\ \Rightarrow &~~ \{n\in \mathbb{N}: \left\|u^{(n)}_{i}- P_{V}(u^{(n)}_{i})\right\|< \epsilon \}\in \mathcal{F(\mathcal{I})}
\\ \Rightarrow &~~ \mathcal{I}-\displaystyle\lim_{n\rightarrow \infty}\left\|u^{(n)}_{i}- P_{V}(u^{(n)}_{i})\right\|= 0 .
\end{align*}
Conversely, let $\mathcal{I}-\displaystyle\lim_{n\rightarrow \infty}\left\|u^{(n)}_{i}- P_{V}(u^{(n)}_{i})\right\|= 0, \forall ~~~ i=1,2,...,k.$
\\Now for $u\in U_{n}= span\{u^{(n)}_{1},u^{(n)}_{2},...,u^{(n)}_{k}\}$ there exists unique scalars $c_{1},c_{2},...,c_{k}\in\mathbb{F}$
such that $ u = c_{1}u^{(n)}_{1}+ c_{2}u^{(n)}_{2}+...+ c_{k}u^{(n)}_{k}$.
So for $u\in U_{n}$ with $\left\|u\right\|= 1$ we have
\begin{align*}
\left\|u- P_{V}(u)\right\|=& \left\|\sum^{k}_{i=1}c_{i}u^{(n)}_{i}- P_{V}(\sum^{k}_{i=1}c_{i}u^{(n)}_{i})\right\|
\\=& \left\|\sum^{k}_{i=1}c_{i}u^{(n)}_{i}- \sum^{k}_{i=1}c_{i}P_{V}(u^{(n)}_{i})\right\|
\\=& \left\|\sum^{k}_{i=1}c_{i}\{u^{(n)}_{i}- P_{V}(u^{(n)}_{i})\}\right\|
\\\leq & \sum^{k}_{i=1}\left|c_{i}\right|\left\|u^{(n)}_{i}- P_{V}(u^{(n)}_{i})\right\|  ~~~~~~~~~~~~~~~~~~~~~~~~~~\ldots \ldots     (2)
\end{align*}
As $\left\|u\right\|=1$, therefore $\displaystyle{\sum^{k}_{i=1}}\left|c_{i}\right|= c > 0$.
\\Let $\epsilon > 0$ be given. Then $\forall~~ i=1,2,...,k$, there exists $B_{i}(\epsilon)\in \mathcal{F(\mathcal{I})}$
such that
\begin{align*}
\left\|u^{(n)}_{i}- P_{V}(u^{(n)}_{i})\right\|< \frac{\epsilon}{2c}, ~~~\forall~~ n\in B_{i}(\epsilon).
\end{align*}
Let $B(\epsilon)= \displaystyle{\bigcap^{k}_{i=1}} B_{i}(\epsilon)$. Then $B(\epsilon)\in \mathcal{F}(\mathcal{I})$. Let $n\in B(\epsilon)$. Then
\begin{align*}
\left\|u^{(n)}_{i}- P_{V}(u^{(n)}_{i})\right\|< \frac{\epsilon}{2c},~~ \forall~~ i=1,2,...,k.
\end{align*}
Then from (2) for $u\in U_{n}$ with $\left\|u\right\|=1$ we have,
\begin{align*}
 \left\|u- P_{V}(u)\right\|&\leq \sum^{k}_{i=1}\left|c_{i}\right|\left\|u^{(n)}_{i}- P_{V}(u^{(n)}_{i})\right\|
\\& < \frac{\epsilon}{2c}. \sum^{k}_{i=1}\left|c_{i}\right|=\frac{\epsilon}{2}
\end{align*}
\begin{align*}
 \Rightarrow  \displaystyle\sup_{u\in U_{n}, \left\|u\right\|=1}\left\|u- P_{V}(u)\right\|\leq \frac{\epsilon}{2}< \epsilon .
\end{align*}
Therefore, $B(\epsilon)\subset \{n\in \mathbb{N}:\displaystyle{\sup_{u\in U_{n},\left\|u\right\|=1}}\left\|u- P_{V}(u)\right\|< \epsilon\}$. Since
$\mathcal B(\epsilon)\in \mathcal F(I)$, so
$\{n\in \mathbb{N}:\displaystyle{\sup_{u\in U_{n},\left\|u\right\|=1}}\left\|u- P_{V}(u)\right\|< \epsilon\} \in \mathcal{F(I)}$. This
implies, $\mathcal{I}-\displaystyle{\lim_{n\rightarrow \infty} \sup_{u\in U_{n},\left\|u\right\|=1}} \left\|u- P_{V}(u)\right\|=0$
i.e. $\mathcal{I}-\displaystyle{\lim_{n\rightarrow \infty}} U_{n}= V$.
\end{proof}

\begin{rem} (i) If $\mathcal{I}$ is an admissible ideal, then the usual convergence of a sequence of $k$-dimensional subspaces
$\{U_{n}\}_{n\in \mathbb N}$ of an inner product space $\mathcal X$, implies the $\mathcal{I}$-convergence of
$\{U_{n}\}_{n\in \mathbb N}$ in $\mathcal X$.

(ii) If we take $\mathcal{I}=\mathcal{I}_{f}=\{A\subset \mathbb{N}$: A is a finite subset of $\mathbb{N}\}$, then
$\mathcal{I}_{f}$-convergence of a sequence $\{U_{n}\}_{n\in \mathbb N}$ of $k$-dimensional subspaces of an inner product
space $\mathcal{X}$ coincides with the usual notion of convergence of sequence $\{U_{n}\}_{n\in \mathbb N}$ of $k$-dimensional
subspaces (\cite{Ma3}).

(iii) Again if we take $\mathcal{I}=\mathcal{I}_{d}=\{A\subset \mathbb{N}: d(A)=0\}$, then $\mathcal{I}_{d}$-convergence
of sequence $\{U_{n}\}_{n\in \mathbb N}$ of $k$-dimensional subspaces of an inner product space $\mathcal X$, coincides
with the notion of statistical convergence
of $\{U_{n}\}_{n\in \mathbb N}$ (\cite{Nu}).
\end{rem}

We now site an example of a sequence of subspaces of an inner product space which is $\mathcal{I}$-convergent
but neither convergent in usual sense nor statistically convergent.

\begin{Example} For each $j\in \mathbb{N}$, let $\mathcal{D}_{j}=\{2^{j-1}(2s-1): s\in \mathbb{N}\}$. Then
$\mathbb{N}= \displaystyle{\bigcup^{\infty}_{j=1}} \mathcal{D}_{j}$, is a decomposition of $\mathbb{N}$, such that $\mathcal{D}_{j}$'s are infinite subsets
of $\mathbb{N}$ and $\mathcal{D}_{i}\cap \mathcal{D}_{j}=\phi$ for $i\neq j$.

Let $\mathcal{I}=\{ A\subset \mathbb{N}$ : A intersects only finitely many $\mathcal{D}_{j}'s\}$. Then $\mathcal{I}$ is
a non-trivial admissible ideal in $\mathbb{N}$. Now we consider the real inner product space $\mathbb{R}^{3}$ with the standard inner product
and let $\{e_{1},e_{2},e_{3}\}$ be the canonical basis of $\mathbb{R}^{3}$.\\
Then $\mathcal{D}_{1}=\{2s-1: s\in \mathbb{N}\}\in \mathcal{I}$
and $\mathcal{B}=\mathbb N- \mathcal{D}_{1}\in
\mathcal{F(\mathcal{I})}$. We considered the sequence
$\{U_{n}\}_{n\in \mathbb N}$ of 1-dimensional subspaces of
$\mathbb{R}^{3}$ defined as follows: $$ U_{n} =
span\{u^{(n)}_{1}=(\frac{1}{n}\sin(n)e_{1}+\frac{1}{n}e_{2})/
\alpha \}, \mbox{ if $n\in \mathcal{B}$  } and  U_{n} =
span\{u^{(n)}_{1}=e_{3}\}, \mbox{  if  $n\in \mathcal{D}_{1}$} $$,
where, $\alpha= \left\|\frac{1}{n}\sin(n)e_{1}+\frac{1}{n}e_{2}\right\|$. Let $V= span\{e_{2}\}$.\\
Let $n\in \mathcal{B}$.
Then the sequence $\{U_{n}\}_{n\in \mathbb N}$ is $\mathcal{I}$-convergent to $V$ but not statistically
convergent to $V$ and hence not convergent in usual sense.
\end{Example}

Thus we can say that $\mathcal{I}$-convergence of sequence of subspaces of an inner product space is a natural generalization
of statistical convergence of sequence of subspaces, which is also a generalization of usual notion of convergence of sequences of subspaces.

%\begin{note} Let $U_{n}~~(n \in\mathbb{N})$ and $V$ be $k$-dimensional subspaces of an inner product space $\mathcal X$. Then for each $i= 1,2,...,k$ we have
%$ \left\|u^{(n)}_{i}- P_{V}(u^{(n)}_{i})\right\|^{2} =
%1- \displaystyle{\sum^{k}_{j=1}}\left|\left\langle  u^{(n)}_{i},v_{j} \right\rangle\right|^{2}. $
%\end{note}

\begin{thm}
Let $U_{n}~~(n \in\mathbb{N})$ and $V$ be $k$-dimensional subspaces of an inner product space $\mathcal X$. Then
$\mathcal{I}-\displaystyle\lim_{n\rightarrow \infty}\left\|u^{(n)}_{i}- P_{V}(u^{(n)}_{i})\right\|= 0$
if and only if $~~\mathcal{I}-\displaystyle\lim_{n\rightarrow \infty}\sum^{k}_{j=1}\left|\left\langle  u^{(n)}_{i},v_{j} \right\rangle\right|^{2} = 1, ~~~\forall~~ i= 1,2,...,k.$
\end{thm}

\begin{proof}
Proof needs simple calculations, so is omitted.
\end{proof}

\begin{cor}
Let $U_{n}~~(n \in\mathbb{N})$ and $V$ be $k$-dimensional subspaces of an inner product space $\mathcal X$. Then the following statements are equivalent:

(i) $\mathcal{I}-\displaystyle{\lim_{n\rightarrow \infty}}U_{n}=V$

(ii) $\mathcal{I}-\displaystyle{\lim_{n\rightarrow \infty}}\left\|u^{(n)}_{i}- P_{V}(u^{(n)}_{i})\right\|= 0,~~~  \forall~~
  i=1,2,...,k$

(iii)  $\mathcal{I}-\displaystyle{\lim_{n\rightarrow \infty}}\sum^{k}_{j=1}\left|\left\langle u^{(n)}_{i},v_{j}\right\rangle
\right|^{2}= 1, ~~~ \forall ~~ i=1,2,...,k. $
\end{cor}

\begin{proof}
The proof directly follows from Theorem 3.1 and Theorem 3.2.
\end{proof}

\begin{note}
Let $U_{n}~~(n \in\mathbb{N})$ and $V$ be $k$-dimensional subspaces of an inner product space $\mathcal X$. For each $i= 1,2,...,k$ we have
\begin{align*}
\left\|P_{V}(u^{(n)}_{i})\right\|^{2} &= \left\langle P_{V}(u^{(n)}_{i}),P_{V}(u^{(n)}_{i}) \right\rangle = \sum^{k}_{j=1}\left|<u^{(n)}_{i},v_{j}> \right|^{2}.
\end{align*}
\end{note}

\begin{thm}
Let $U_{n}~~(n \in\mathbb{N})$ and $V$ be $k$-dimensional subspaces of an inner product space $\mathcal X$. Then
$\mathcal{I}-\displaystyle{\lim_{n\rightarrow \infty}}\sum^{k}_{j=1}\left|<u^{(n)}_{i},v_{j}> \right|^{2}=1$
if and only if $~~\mathcal{I}-\displaystyle{\lim_{n\rightarrow \infty}} \left\|P_{V}(u^{(n)}_{i})\right\|= 1$ for all $i=1,2,...,k.$
\end{thm}

\begin{proof}
Fix $i\in \{1,2,...,k\}$. Let $\mathcal{I}-\displaystyle{\lim_{n\rightarrow \infty}}\sum^{k}_{j=1}\left|<u^{(n)}_{i},v_{j}> \right|^{2}=1.$ Let
$\epsilon > 0$ be given. Then there exists $\mathcal{A}(\epsilon)\in
\mathcal{F}(\mathcal{I})$ such that
\begin{align*}
& \left|\sum^{k}_{j=1}\left|<u^{(n)}_{i},v_{j}> \right|^{2} - 1\right| < \epsilon, ~~~ \forall ~~ n\in \mathcal{A}(\epsilon)
\\\Rightarrow & \left|\left\|P_{V}(u^{(n)}_{i})\right\|^{2}- 1\right| < \epsilon, ~~~ \forall ~~ n\in \mathcal{A}(\epsilon)
\\\Rightarrow & \left|\left\|P_{V}(u^{(n)}_{i})\right\|- 1\right| < \frac{\epsilon}{1+ \left\|P_{V}(u^{(n)}_{i})\right\| }
< \epsilon, ~~~  \forall ~~  n\in \mathcal{A}(\epsilon).
\end{align*}
Then, $ \left|||P_{V}(u^{(n)}_{i})||- 1\right| < \epsilon, \forall n\in \mathcal{A}(\epsilon).$ This gives $\mathcal{A}(\epsilon)
\subset \left\{ n\in \mathbb{N}: \left| ||P_{V}(u^{(n)}_{i}) ||- 1\right| < \epsilon \right\} .$ Consequently,
$ \left\{ n\in \mathbb{N}: \left| ||P_{V}(u^{(n)}_{i})||- 1\right| < \epsilon \right\} \in
\mathcal{F}(\mathcal{I}) $ and so $\mathcal{I}-\displaystyle{\lim_{n\rightarrow \infty}}
\left\|P_{V}(u^{(n)}_{i})\right\|= 1$.

Conversely, let $\mathcal{I}-\displaystyle{\lim_{n\rightarrow \infty}}
\left\|P_{V}(u^{(n)}_{i})\right\|= 1$. Let $0 < \epsilon \leq 1 $. Then there exists $\mathcal{B}(\epsilon)\in \mathcal{F}(\mathcal{I})$ such that
\begin{align*}
& \left| ||P_{V}(u^{(n)}_{i}) ||- 1\right| < \frac{\epsilon}{3}, ~~~   \forall~~ n\in \mathcal{B(\epsilon)}
\\ \mbox{and} & \left\|P_{V}(u^{(n)}_{i})\right\|+ 1 < 3, ~~~ \forall~~ n\in \mathcal{B(\epsilon)}.
\end{align*}
Then, for each $n\in \mathcal{B}(\epsilon)$
\begin{align*}
 &\left| ||P_{V}(u^{(n)}_{i})||^{2}- 1\right|  = \left(||P_{V}(u^{(n)}_{i}) ||+ 1\right) \left| ||P_{V}(u^{(n)}_{i})||- 1\right|
 < 3 \frac{\epsilon}{3} = \epsilon
\\\Rightarrow & \left|\sum^{k}_{j=1}\left|<u^{(n)}_{i},v_{j}> \right|^{2}- 1\right| < \epsilon.
\end{align*}
Then, $\mathcal{B}(\epsilon)\subset \left\{n\in \mathbb{N}:
\left|\sum^{k}_{j=1}\left|<u^{(n)}_{i},v_{j}> \right|^{2}- 1\right|
< \epsilon \right\}.$ Consequently,
\begin{equation}
\left\{n\in \mathbb{N}:\left|\sum^{k}_{j=1}\left|<u^{(n)}_{i},v_{j}> \right|^{2}- 1\right|< \epsilon \right\}\in
\mathcal{F(\mathcal{I})}.
\end{equation}
Now let $\epsilon > 1.$ Then,
\begin{align*}
\left\{n\in \mathbb{N}:\left|\sum^{k}_{j=1}\left|<u^{(n)}_{i},v_{j}> \right|^{2}- 1\right|
< 1 \right\}\subset \left\{n\in \mathbb{N}:\left|\sum^{k}_{j=1}\left|<u^{(n)}_{i},v_{j}> \right|^{2}- 1\right| < \epsilon \right\}.
\end{align*}
Then using (1) we have
$\left\{n\in \mathbb{N}:\left|\sum^{k}_{j=1}\left|<u^{(n)}_{i},v_{j}> \right|^{2}- 1\right| < \epsilon \right\} \in \mathcal{F}(\mathcal{I})$.
Thus for any $\epsilon > 0, \left\{n\in \mathbb{N}:\left|\sum^{k}_{j=1}\left|<u^{(n)}_{i},v_{j}> \right|^{2}- 1\right|
< \epsilon \right\} \in \mathcal{F}(\mathcal{I})$. Hence, $\mathcal{I}-\displaystyle{\lim_{n\rightarrow \infty}} \sum^{k}_{j=1}\left|<u^{(n)}_{i}, v_{j}>\right|^{2}= 1$.
\end{proof}

\begin{cor}
Let $U_{n}~~(n \in\mathbb{N})$ and $V$ be $k$-dimensional subspaces of an inner product space $\mathcal X$. Then the following statements are equivalent:

(i) $\mathcal{I}-\displaystyle{\lim_{n\rightarrow \infty}}
U_{n}= V $

(ii) $\mathcal{I}-\displaystyle{\lim_{n\rightarrow \infty}}
\left\|u^{(n)}_{i}- P_{V}(u^{(n)}_{i})\right\|= 0, \forall
i = 1, 2, ..., k $

(iii) $\mathcal{I}-\displaystyle{\lim_{n\rightarrow \infty}}
\sum^{k}_{j=1}\left|<u^{(n)}_{i}, v_{j}>\right|^{2}= 1,
\forall i = 1, 2, ..., k $

(iv) $\mathcal{I}-\displaystyle{\lim_{n\rightarrow \infty}}
\left\| P_{V}(u^{(n)}_{i})\right\|=1, \forall i = 1, 2, ..., k $.
\end{cor}

\begin{proof}
The proof directly follows from Theorem 3.1, Theorem 3.2 and Theorem 3.4.
\end{proof}

Following Gunawan et al. \cite{Gu}, we recall the notion of the standard $n$-norm of $n$-vectors in an inner product
space $\mathcal{X}$ over $\mathcal F ( = \mathbb{R} ~\mbox{or}~ \mathbb{C})$ of dimension $k$ or higher.

The standard $n$-inner product of $n$-vectors $x_{1},x_{2},\ldots,x_{n}$ is defined as
\begin{align*}
G(x_{1},x_{2},\ldots,x_{n})= \left|
\begin{array}{l l l l}
<x_{1},x_{1}> & <x_{1},x_{2}> & \ldots & <x_{1},x_{n}>\\
<x_{2},x_{1}> & <x_{2},x_{2}> & \ldots & <x_{2},x_{n}>\\
\vdots & \vdots && \vdots \\
<x_{n},x_{1}> & <x_{n},x_{2}> & \ldots & <x_{n},x_{n}>
\end{array}
\right|,
\end{align*}
here $G(x_{1},x_{2},\ldots,x_{n})$ is called the Gramian of the vectors $ x_{1},x_{2},\ldots,x_{n} $.

Clearly the vectors $x_{1},x_{2},\ldots,x_{n}$ in $\mathcal{X}$
are linearly dependent if and only if the Gramian of the $n$-vectors
i.e. $G(x_{1},x_{2},\ldots,x_{n})$ vanishes.

The standard $n$-norm is defined as
\begin{align*}
\left\|x_{1},x_{2},\ldots,x_{n}\right\|= \sqrt{G(x_{1},x_{2},\ldots,x_{n})}.
\end{align*}

\begin{thm}
Let $U_{n}~~(n \in\mathbb{N})$ and $V$ be $k$-dimensional subspaces of an inner product space $\mathcal X$. If
$\mathcal{I}-\displaystyle{\lim_{n \rightarrow \infty}}U_{n}= V, $ then
$\mathcal{I}-\displaystyle{\lim_{n\rightarrow \infty}} \left\|u^{(n)}_{i}, P_{V}(u^{(n)}_{i})\right\|
= 0, \forall i = 1, 2, \ldots, k $.
\end{thm}

\begin{proof}
The proof needs simple calculations, so is omitted.
\end{proof}
The Converse of the above theorem is not true. To show this we consider the following example.

\begin{Example}
Let $\mathcal{I}$ be a non-trivial admissible ideal of $\mathbb{N}$. Consider the
real vector space $\mathcal{X}= \mathbb{R}^{2}$ with standard inner product. Let $\{e_{1}, e_{2}\}$ be the canonical basis of $\mathbb{R}^{2}$.

We consider the sequence $\{ U_{n}\}_{n\in \mathbb{N}}$ of one-dimensional subspaces of $\mathcal{X}$, defined as follows:
\begin{align*}
& U_{n}= span\{u^{(n)}_{1}\},~~ \mbox{where}~~ u^{(n)}_{1}= e_{1}, ~~~ n \in \mathbb{N}
\\ \mbox{and}~~ & V= span \{e_{2}\}.
\end{align*}
Then $\mathcal{I}-\displaystyle{\lim_{n\rightarrow \infty}}\left\| u^{(n)}_{1}, P_{V}(u^{(n)}_{1})\right\|= 0 $ but
the sequence $\{U_{n}\}_{n\in \mathbb{N}}$ is not $\mathcal{I}-$convergent to the
subspace $V$ of $\mathcal X$.
\end{Example}

\begin{thm}
Let $U_{n}~~(n \in\mathbb{N})$ and $V$ be $k$-dimensional subspaces of an inner product space $\mathcal X$. Then
$\mathcal{I}-\displaystyle{\lim_{n\rightarrow \infty}}\left\|u^{(n)}_{i},v_{1}, v_{2},\ldots, v_{k} \right\|= 0$ if
and only if $\mathcal{I}-\displaystyle{\lim_{n\rightarrow \infty}}\sum^{k}_{j=1}\left|<u^{(n)}_{i}, v_{j} >\right|^{2}= 1$ for
all $i=1,2,...,k$.
\end{thm}

\begin{proof}
Fix $i\in \{1,2,...,k\}$. Then using mathematical induction on $k$, we have
\begin{align*}
\left\|u^{(n)}_{i},v_{1}, v_{2},\ldots, v_{k} \right\|^{2}= 1-
\sum^{k}_{j=1}\left|<u^{(n)}_{i}, v_{j} >\right|^{2}.
\end{align*}
First, let $\mathcal{I}-\displaystyle{\lim_{n\rightarrow \infty}}\left\|u^{(n)}_{i},v_{1}, v_{2},\ldots, v_{k} \right\|= 0$ and
let $\epsilon > 0$ be given. Then there exists $\mathcal{A(\epsilon)} \in \mathcal{F(I)}$ such that,
\begin{align*}
& \left\|u^{(n)}_{i},v_{1}, v_{2},\ldots, v_{k} \right\| < \sqrt{\epsilon}, ~~~ \forall ~~~ n\in \mathcal{A(\epsilon)} \\
\Rightarrow & \left\|u^{(n)}_{i},v_{1}, v_{2},\ldots, v_{k} \right\|^{2} < \epsilon, ~~~ \forall ~~~ n\in \mathcal{A(\epsilon)} \\
\Rightarrow & \left|\sum^{k}_{j=1}\left|<u^{(n)}_{i}, v_{j} >\right|^{2} - 1 \right| < \epsilon, ~~~ \forall ~~~ n\in \mathcal{A(\epsilon)}.
\end{align*}
Then we have,
\begin{align*}
& \mathcal{A(\epsilon)}\subset \left\{n\in \mathbb{N}: \left|\sum^{k}_{j=1}\left|<u^{(n)}_{i}, v_{j} >\right|^{2} - 1 \right|< \epsilon \right\} \\
\Rightarrow & \left\{n\in \mathbb{N}: \left|\sum^{k}_{j=1}\left|<u^{(n)}_{i}, v_{j} >\right|^{2} - 1 \right|< \epsilon \right\} \in \mathcal{F(I)}.
\end{align*}
Consequently, $\mathcal{I}-\displaystyle{\lim_{n\rightarrow \infty}}\sum^{k}_{j=1}\left|<u^{(n)}_{i}, v_{j} >\right|^{2}= 1.$

Conversely, let $\mathcal{I}-\displaystyle{\lim_{n\rightarrow \infty}}\sum^{k}_{j=1}\left|<u^{(n)}_{i}, v_{j} >\right|^{2}= 1.$ Let $\epsilon > 0$ be given.
Then there exists $\mathcal{B(\epsilon)} \in \mathcal{F(I)}$ such that,
\begin{align*}
& \left|\sum^{k}_{j=1}\left|<u^{(n)}_{i}, v_{j} >\right|^{2} - 1 \right| < \epsilon^{2}, ~~~~ \forall ~~~ n\in \mathcal{B(\epsilon)}\\
\Rightarrow & \left\|u^{(n)}_{i},v_{1}, v_{2},\ldots, v_{k} \right\|^{2} < \epsilon^{2}, ~~~~ \forall ~~~ n\in \mathcal{B(\epsilon)} \\
\Rightarrow & \left\|u^{(n)}_{i},v_{1}, v_{2},\ldots, v_{k} \right\| < \epsilon, ~~~~ \forall ~~~ n\in \mathcal{B(\epsilon)}.
\end{align*}
Then we have,
\begin{align*}
& \mathcal{B(\epsilon)}\subset \left\{n\in \mathbb{N}:\left\|u^{(n)}_{i},v_{1}, v_{2},\ldots, v_{k} \right\| < \epsilon \right\} \\
\Rightarrow & \left\{n\in \mathbb{N}:\left\|u^{(n)}_{i},v_{1}, v_{2},\ldots, v_{k} \right\| < \epsilon \right\}\in \mathcal{F(I)} \\
\Rightarrow & ~~ \mathcal{I}-\displaystyle{\lim_{n\rightarrow \infty}}\left\|u^{(n)}_{i},v_{1}, v_{2},\ldots, v_{k} \right\|= 0.
\end{align*}
\end{proof}

\begin{cor}
Let $U_{n}~~(n \in\mathbb{N})$ and $V$ be $k$-dimensional subspaces of an inner product space $\mathcal X$. Then the following statements are equivalent:

(i) $\mathcal{I}-\displaystyle{\lim_{n\rightarrow \infty}}U_{n}= V$

(ii) $\mathcal{I}-\displaystyle{\lim_{n\rightarrow \infty}}\left\|u^{(n)}_{i} - P_{V}(u^{(n)}_{i})\right\|= 0, \forall i = 1, 2, ..., k$

(iii) $\mathcal{I}-\displaystyle{\lim_{n\rightarrow \infty}} \sum^{k}_{j=1}\left|<u^{(n)}_{i}, v_{j}>\right|^{2}= 1, \forall i = 1, 2, ..., k$

(iv) $\mathcal{I}-\displaystyle{\lim_{n\rightarrow \infty}}\left\|P_{V}(u^{(n)}_{i})\right\|= 1, \forall 1 = 1, 2, ..., k $

(v) $\mathcal{I}-\displaystyle{\lim_{n\rightarrow \infty}}\left\|u^{(n)}_{i}, v_{1},v_{2},..., v_{k}\right\|=0, \forall i = 1, 2, ..., k.$
\end{cor}

\begin{proof}
The proof directly follows from Theorem 3.1, Theorem 3.2, Theorem 3.4 and Theorem 3.7.
\end{proof}

\noindent\textbf{Acknowledgement:}
The second author is grateful to University Grants Comission, India for financial support under UGC-JRF scheme during the preparation of this paper.
\\
 %----------------------REFERENCES-------------------------------------

\end{document}